\documentclass[11pt]{amsart}

\usepackage[T1]{fontenc}
\usepackage[utf8]{inputenc}
\usepackage{geometry}
\usepackage{amssymb}
\usepackage{amsmath}
\usepackage{mathtools}
\usepackage{xcolor}
\usepackage[colorlinks=true,citecolor=red,linkcolor=blue,urlcolor=red]{hyperref}

\DeclareMathOperator{\rank}{rank}

\newtheorem{thm}{Theorem}[section]

\newtheorem{prop}[thm]{Proposition}

\theoremstyle{definition}

\newtheorem{rem}[thm]{Remark}

\numberwithin{equation}{section}

\begin{document}

\title[A counterexample to the odd-dimensional rank bound]{A counterexample to the odd-dimensional rank bound for abelian \texorpdfstring{$p$}{p}-group actions}

\author{Jihao Liu}
\author{Yanze Wang}

\address{Department of Mathematics, Peking University, No. 5 Yiheyuan Road, Haidian District, Beijing 100871, China}
\address{Beijing International Center for Mathematical Research, Peking University, No. 5 Yiheyuan Road, Haidian District, Beijing 100871, China}
\email{liujihao@math.pku.edu.cn}

\address{Academy of Mathematics and Systems Science, Chinese Academy of Sciences, No. 55 Zhonguancun East Road, Haidian District, Beijing, 100190, China}
\email{wangyanze@amss.ac.cn}

\date{\today}

\begin{abstract}
  We disprove the odd-dimensional extension, conjectured by Moraga and recorded by Koll\'ar and Zhuang, of the rank bound for faithful abelian $p$-group actions on smooth Calabi--Yau varieties.  The main result of this paper was obtained by ChatGPT 5.5 pro, and the Danus system based on the Rethlas system.
\end{abstract}

\maketitle

\section{Introduction}\label{sec:introduction}

A recurring theme in the study of finite group actions on algebraic varieties is to bound the size of an abelian symmetry group in terms of the dimension of the variety on which it acts. For actions of abelian $p$-groups, the natural measure of size is the rank, the minimal number of generators; equivalently, for an elementary abelian $p$-group it is the dimension over the field $\mathbb F_p$ with $p$ elements.

Koll\'ar and Zhuang \cite[Corollary~9]{KZ26} study the essential dimension of isogenies and, in the course of that work, record a rank bound for faithful abelian $p$-group actions. For a finite abelian $p$-group $G$ acting faithfully on a smooth projective variety $X$ of dimension $n$ over a field of characteristic zero, their Corollary~9 yields, under the hypothesis that $\chi(X,\mathcal O_X)=\pm 1$ or $\pm 2$, the bound
\begin{equation}\label{eq:kz-bound}
  \rank G\leq
  \begin{cases}
    \dfrac{p}{p-1}\,n, & p\text{ odd},\\[2mm]
    2n+1, & p=2.
  \end{cases}
\end{equation}
As Koll\'ar and Zhuang observe, their method does not control the rank when $\chi(X,\mathcal O_X)=0$, which is precisely the situation of a Calabi--Yau threefold, yet Moraga conjectured that the bounds \eqref{eq:kz-bound} hold in odd dimensions as well \cite[Remark 22]{KZ26}. In particular, the bound should hold for Calabi-Yau threefolds. The purpose of this note is to show that the odd-prime case of this conjectural odd-dimensional extension is false.

\begin{thm}\label{thm:main}
  Let $X\subset\mathbb P^4_{\mathbb C}$ be the Fermat quintic hypersurface defined by $x_0^5+x_1^5+x_2^5+x_3^5+x_4^5=0$, and let $G=(\mu_5)^5/\Delta$ act on $X$ by diagonal coordinate scalings, where $\mu_5$ is the group of complex fifth roots of unity and $\Delta$ is its scalar diagonal subgroup. Then $X$ is a smooth projective strict Calabi--Yau threefold over $\mathbb C$, the group $G$ is an elementary abelian $5$-group of rank $4$ acting faithfully on $X$, and, for $p=5$ and $n=\dim X=3$,
  \begin{equation}\label{eq:main}
    \rank G=4>\frac{15}{4}=\frac{p}{p-1}\,n.
  \end{equation}
  In particular, the conjecture of Moraga in \cite[Remark~22]{KZ26} is false.
\end{thm}

\begin{rem}
The sketch of the proof of the main result of this paper was obtained by Chatgpt 5.5 pro, and later summed up, verified, and properly written by the Danus system, a specialized agent built on Rethlas and substantially more capable for fundamental mathematical research based on the Rethlas system. Human verification and polishing were done afterwards. See \cite{Ju26} for a detailed introduction to the Rethlas system. Due to the limitation of automated systems, it is possible that we have missed some related references in the literature, and we welcome any comments from experts.
\end{rem}

\subsection*{Outline of the argument}
The proof is entirely elementary. The Fermat quintic $X$ is a smooth degree-$5$ hypersurface in $\mathbb P^4$, so it is a smooth projective threefold, and adjunction together with the hypersurface exact sequence shows that it is a strict Calabi--Yau threefold (Section~\ref{sec:variety}). The group $(\mu_5)^5$ rescaling the five homogeneous coordinates by fifth roots of unity preserves the Fermat equation, the scalar diagonal acts trivially on projective space, and the quotient $G$ is an elementary abelian $5$-group of rank $4$ (Section~\ref{sec:group}). Testing on the two-coordinate points of $X$ forces every coordinate scaling that acts trivially on $X$ to be scalar, so the induced $G$-action is faithful. Finally, the arithmetic inequality $4>15/4$ furnishes the violation of \eqref{eq:kz-bound} and completes the proof of Theorem~\ref{thm:main} (Section~\ref{sec:proof}).

\subsection*{Acknowledgements}
The first author was partially supported by the National Key R\&D Program of China \#\allowbreak 2024YFA1014400. The first author would like to thank the Rethlas team, namely Haocheng Ju, Jiedong Jiang, Shurui Liu, Guoxiong Gao, Yuefeng Wang, Zeming Sun, Bin Wu, Liang Xiao, and Bin Dong, for their contributions to the development of Rethlas and its customized version used for the problem studied in this paper. The first author would like to thank Ruochuan Liu and Gang Tian for constant support and encouragement. The second author would like to thank Yifei Chen for constant support.

\section{The Fermat quintic threefold}\label{sec:variety}

Throughout, we work over the complex numbers $\mathbb C$. Let $\mathbb P^4_{\mathbb C}$ have homogeneous coordinates $x_0,x_1,x_2,x_3,x_4$, write
\[
  F=x_0^5+x_1^5+x_2^5+x_3^5+x_4^5,
\]
and let $X=V(F)\subset\mathbb P^4_{\mathbb C}$ be its zero locus. In this section we record that $X$ is a smooth projective strict Calabi--Yau threefold. This is well-known but our system still writes a proof of it so we record the proof here.

\begin{prop}\label{prop:cy}
  The variety $X$ is a smooth projective threefold over $\mathbb C$ with $\omega_X\cong\mathcal O_X$ and
  \[
    H^1(X,\mathcal O_X)=H^2(X,\mathcal O_X)=0.
  \]
\end{prop}

\begin{proof}
  The polynomial $F$ is homogeneous, so $X=V(F)$ is a closed subvariety of $\mathbb P^4_{\mathbb C}$ and hence projective over $\mathbb C$. As $F$ is a single nonzero homogeneous polynomial, $X$ is a hypersurface, so once smoothness is established every irreducible component has dimension $4-1=3$.

  We prove smoothness by the Jacobian criterion. The partial derivatives of $F$ are
  \[
    \partial F/\partial x_i=5x_i^4\qquad(0\leq i\leq 4).
  \]
  Since the characteristic of $\mathbb C$ is zero, the scalar $5$ is nonzero. If all five partial derivatives vanished at a point represented by a vector $(x_0,\dots,x_4)$, then $x_i=0$ for every $i$, which is impossible for a point of projective space. Hence $X$ has no singular point, so $X$ is an irreducible smooth threefold.

  We next verify the strict Calabi--Yau conditions. Since $X$ is a smooth hypersurface of degree $5$ in $\mathbb P^4$, the adjunction formula gives
  \[
    \omega_X\cong\bigl(\omega_{\mathbb P^4}\otimes\mathcal O_{\mathbb P^4}(5)\bigr)\big|_X.
  \]
  Because $\omega_{\mathbb P^4}\cong\mathcal O_{\mathbb P^4}(-5)$, this yields $\omega_X\cong\mathcal O_X$. For the cohomology, consider the hypersurface exact sequence
  \begin{equation}\label{eq:hypersurface-ses}
    0\to\mathcal O_{\mathbb P^4}(-5)\xrightarrow{\ \cdot F\ }\mathcal O_{\mathbb P^4}\to\mathcal O_X\to 0.
  \end{equation}
  The standard cohomology of line bundles on projective space gives
  \[
    H^1(\mathbb P^4,\mathcal O_{\mathbb P^4})=H^2(\mathbb P^4,\mathcal O_{\mathbb P^4})=0
  \]
  and
  \[
    H^2(\mathbb P^4,\mathcal O_{\mathbb P^4}(-5))=H^3(\mathbb P^4,\mathcal O_{\mathbb P^4}(-5))=0.
  \]
  The long exact cohomology sequence associated with \eqref{eq:hypersurface-ses} then yields $H^1(X,\mathcal O_X)=0$ and $H^2(X,\mathcal O_X)=0$. Thus $X$ is a smooth projective strict Calabi--Yau threefold over $\mathbb C$.
\end{proof}

\section{The group and its action}\label{sec:group}

Let $\mu_5=\{\lambda\in\mathbb C^\times:\lambda^5=1\}$ be the group of complex fifth roots of unity, and let
\[
  D=(\mu_5)^5
\]
act on $\mathbb P^4_{\mathbb C}$ by
\begin{equation}\label{eq:torus-action}
  (\lambda_0,\dots,\lambda_4)\cdot[x_0:\cdots:x_4]=[\lambda_0x_0:\cdots:\lambda_4x_4].
\end{equation}
Let
\[
  \Delta=\{(\lambda,\lambda,\lambda,\lambda,\lambda):\lambda\in\mu_5\}\subset D
\]
be the scalar diagonal subgroup, and set $G=D/\Delta$.

\begin{prop}\label{prop:group}
  The action \eqref{eq:torus-action} preserves $X$ and the subgroup $\Delta$ acts trivially on $X$, so the action descends to an action of $G=D/\Delta$ on $X$. The group $G$ is an elementary abelian $5$-group isomorphic to $(\mathbb Z/5\mathbb Z)^4$, and $\rank G=4$.
\end{prop}

\begin{proof}
  For every $(\lambda_0,\dots,\lambda_4)\in D$ one has $\lambda_i^5=1$ for each $i$, so
  \[
    F(\lambda_0x_0,\dots,\lambda_4x_4)=\sum_{i=0}^4\lambda_i^5x_i^5=\sum_{i=0}^4x_i^5=F(x_0,\dots,x_4).
  \]
  Hence the action \eqref{eq:torus-action} preserves $X$. Every element of $\Delta$ multiplies all homogeneous coordinates by the same nonzero scalar, hence acts trivially on $\mathbb P^4_{\mathbb C}$ and in particular on $X$. Therefore the action of $D$ on $X$ descends to an action of $G=D/\Delta$ on $X$.

  Since $\mu_5$ is cyclic of order $5$, the group $D=(\mu_5)^5$ is isomorphic to $(\mathbb Z/5\mathbb Z)^5$, and under this identification $\Delta$ is the one-dimensional subgroup generated by $(1,1,1,1,1)$. The homomorphism
  \[
    (\lambda_0,\dots,\lambda_4)\longmapsto(\lambda_1/\lambda_0,\lambda_2/\lambda_0,\lambda_3/\lambda_0,\lambda_4/\lambda_0)
  \]
  from $D$ to $(\mu_5)^4$ is surjective with kernel exactly $\Delta$, so $G\cong(\mathbb Z/5\mathbb Z)^4$. In particular $G$ is an elementary abelian $5$-group, and its rank, the minimal number of generators, equals $4$.
\end{proof}

We now prove that the induced action of $G$ on $X$ is faithful. The point of the argument is that the kernel of the $D$-action on $X$ is no larger than the scalar diagonal $\Delta$; two-coordinate points of $X$ detect every non-scalar diagonal element.

\begin{prop}\label{prop:faithful}
  The kernel of the action of $D$ on $X$ is exactly $\Delta$. Consequently, the induced action of $G=D/\Delta$ on $X$ is faithful.
\end{prop}

\begin{proof}
  By Proposition~\ref{prop:group}, $\Delta$ acts trivially on $X$, so $\Delta$ is contained in the kernel of the $D$-action. Conversely, suppose $d=(\lambda_0,\dots,\lambda_4)\in D$ acts trivially on $X$.

  Fix two distinct indices $i,j\in\{0,1,2,3,4\}$. Let $P_{ij}\in\mathbb P^4_{\mathbb C}$ be the point whose $i$-th coordinate is $1$, whose $j$-th coordinate is $-1$, and whose remaining coordinates are $0$. Then
  \[
    F(P_{ij})=1^5+(-1)^5=0,
  \]
  so $P_{ij}\in X$. Since $d$ acts trivially on $X$, we have $d\cdot P_{ij}=P_{ij}$ in $\mathbb P^4_{\mathbb C}$. Reading off the two nonzero coordinates, this equality says that the points
  \[
    [\lambda_i:-\lambda_j]\quad\text{and}\quad[1:-1]
  \]
  of $\mathbb P^1$ coincide. Thus $\lambda_i=\lambda_j$.

  The pair $i,j$ of distinct indices is arbitrary, so all five entries $\lambda_0,\dots,\lambda_4$ are equal. Therefore $d\in\Delta$, and the kernel of the $D$-action on $X$ is exactly $\Delta$. Since $G=D/\Delta$, the induced $G$-action on $X$ is faithful.
\end{proof}

\section{Proof of Theorem~\ref{thm:main}}\label{sec:proof}

\begin{proof}[Proof of Theorem~\ref{thm:main}]
  By Proposition~\ref{prop:cy}, the Fermat quintic $X=V(x_0^5+\cdots+x_4^5)\subset\mathbb P^4_{\mathbb C}$ is a smooth projective strict Calabi--Yau threefold over $\mathbb C$; in particular $\dim X=3$. By Proposition~\ref{prop:group}, the diagonal action of $D=(\mu_5)^5$ preserves $X$ and descends to an action of $G=D/\Delta$ on $X$, where $G\cong(\mathbb Z/5\mathbb Z)^4$ is an elementary abelian $5$-group with $\rank G=4$. By Proposition~\ref{prop:faithful}, this action of $G$ on $X$ is faithful.

  Take $p=5$ and $n=\dim X=3$. The bound \eqref{eq:kz-bound} predicted by Moraga in \cite[Remark~22]{KZ26} would give
  \[
    \rank G\leq\frac{p}{p-1}\,n=\frac{5}{4}\cdot 3=\frac{15}{4}.
  \]
  But $\rank G=4$, and $4>15/4$, which establishes \eqref{eq:main}. Hence the faithful action of $G$ on the smooth projective variety $X$ of odd dimension $n=3$ violates \eqref{eq:kz-bound} for the odd prime $p=5$, so  the conjecture of Moraga in \cite[Remark~22]{KZ26} is false.
\end{proof}

\begin{rem}\label{rem:scope}
The counterexample is a smooth projective strict Calabi--Yau threefold, so $\chi(X,\mathcal O_X)=0$ and the bound \eqref{eq:kz-bound} of \cite[Remark~22]{KZ26} is not available from their method. Theorem~\ref{thm:main} disproves the odd-prime case of the conjectural odd-dimensional extension recorded in \cite[Remark~22]{KZ26}; we make no claim about the $p=2$ case, nor about whether a modified rank bound might hold in odd dimensions.
\end{rem}

\begin{rem}
  In fact, for any odd prime $p$, we can take the hypersurface $X=V(\sum_{i=0}^{p-1} x_i^p)$ in $\mathbb{P}^{p-1}_{\mathbb{C}}$ and the group $G=(\mu_p)^{p}/\Delta \cong (\mathbb{Z} / p{\mathbb{Z} })^{p-1}$. Then we have $\rank G=p-1$, $n=\dim X=p-2$, and
\[
  \rank G-\frac{p}{p-1}n=(p-1)-\frac{p(p-2)}{p-1}=\frac{1}{p-1}>0,
\]
so $\rank G > \frac{p}{p-1}n$.
\end{rem}

\end{document}